\documentclass[12pt]{article}
\usepackage[top=2.8cm, left=2.5cm, right=2.5cm, bottom=2.3cm]{geometry}
\newcommand{\minitab}[2][l]{\begin{tabular}#1 #2\end{tabular}}

\usepackage{amsmath,multirow,mathtools}
\usepackage{caption}

\usepackage{amssymb}
\usepackage{amsthm}
\usepackage{epsf,epsfig,amsfonts,graphicx}

\usepackage{listings}
\usepackage{xcolor}

\definecolor{codegreen}{rgb}{0,0.6,0}
\definecolor{codegray}{rgb}{0.5,0.5,0.5}
\definecolor{codepurple}{rgb}{0.58,0,0.82}
\definecolor{backcolour}{rgb}{0.95,0.95,0.92}

\lstdefinestyle{mystyle}{
	backgroundcolor=\color{backcolour},   
	commentstyle=\color{codegreen},
	keywordstyle=\color{magenta},
	numberstyle=\tiny\color{codegray},
	stringstyle=\color{codepurple},
	basicstyle=\ttfamily\footnotesize,
	breakatwhitespace=false,         
	breaklines=true,                 
	captionpos=b,                    
	keepspaces=true,                 
	numbers=left,                    
	numbersep=5pt,                  
	showspaces=false,                
	showstringspaces=false,
	showtabs=false,                  
	tabsize=2
}

\theoremstyle{definition}
\newtheorem{mydef}{Definition}[section]

\newtheorem{myrem}{Remark}[section]

\theoremstyle{plain}
\newtheorem{mytheo}{Theorem}[section]
\newtheorem{mycor}{Corollary}[section]
\newtheorem{myprop}{Proposition}[section]

\title{Classification of differentially non-degenerate left-symmetric algebras in dimension 3}
\author{Serena Scapucci\footnote{Institut f\"ur Mathematik, Friedrich Schiller Universit\"at Jena, 07743 Jena Germany, serena.scapucci@uni-jena.de}}
\date{}

\begin{document}
	
	\maketitle
	
\begin{abstract}
	 In terms of Nijenhuis geometry, left-symmetric algebras are the same as Nijenhuis operators whose entries are linear in coordinates. Here we consider Nijenhuis operators for which the coefficients of its characteristic polynomial are algebraically independent. We give a classification of this type of operators and hence of the corresponding LSAs (which we will call differentially non-degenerate LSA) in dimension 3.
\end{abstract}

\section{Introduction}

We recall some basic definitions in the field of Nijenhuis Geometry and left-symmetric algebras (LSA). See also \cite{NijGeo}. 
\subsection{Nijenhuis Geometry} 

Let $L$ be a $(1,1)$-tensor field on a smooth manifold $M^n$ of dimension $n$. Then one can define the Nijenhuis torsion $\mathcal{N}_L$ of the operator $L$ in several equivalent ways.
\begin{mydef}
	For a pair of vector fields $\xi$ and $\eta$, the \textit{Nijenhuis torsion} $\mathcal{N}_L$ is defined by the following formula:
	\begin{center}
		$\mathcal{N}_L(\xi,\eta):=L^2[\xi,\eta]+[L\xi,L\eta]-L[L\xi,\eta]-L[\xi,L\eta]$.
	\end{center}
\end{mydef}
\begin{mydef}
	In local coordinates $x^1,\dots,x^n$, the components $(\mathcal{N}_L)_{jk}^i$ of $\mathcal{N}_L$ are defined as follows:
	\begin{center}
		$(\mathcal{N}_L)_{jk}^i:=L^s_j\frac{\partial L^i_k}{\partial x^s}-L^s_k\frac{\partial L^i_j}{\partial x^s}-L^i_s\frac{\partial L^s_k}{\partial x^j}+L^i_s\frac{\partial L^s_j}{\partial x^k}$
	\end{center}
	where $L^i_j$ denote the components of $L$.
\end{mydef}
The latter is most convenient for computations. 

\begin{mydef}
	A $(1,1)$-tensor field $L$ defined on a smooth manifold $M^n$ is called \textit{Nijenhuis operator} if its Nejenhuis torsion vanishes, i.e. if
	\begin{center}$\mathcal{N}_L=0.$\end{center} 
\end{mydef}
We consider Nijenhuis operators with singular points. At each point $p\in M$, we can define the \textit{algebraic type} of $L(p):T_pM\to T_pM$ as the structure of its Jordan canonical form, which is characterized by the sizes of the Jordan blocks related to each eigenvalue $\lambda_i$ of $L(p)$. 

\begin{mydef}
	A point $p\in M$ is called \textit{algebraically generic}, if the algebraic type of $L$ does not change in some neighbourhood of $p$ in $M$. (We will also say that $L$ is algebraically generic at $p\in M$).
	A point $p\in M$ is called \textit{singular}, if it is not algebraically generic.
\end{mydef}

One more notion is related to analytic properties of the coefficients of the characteristic polynomial $\chi_{L(x)}(t)=\det(t\cdot\text{Id}-L(x))=\sum_{k=0}^{n}t^k\sigma_{n-k}(x)$ of the operator $L$.

\begin{mydef}
	A point $x\in M$ is called \textit{differentially non-degenerate}, if the differentials d$\sigma_1(x),\dots,$d$\sigma_n(x)$ of the coefficients of the characteristic polynomial $\chi_{L(x)}(t)$ are linearly independent at this point.
\end{mydef} 

In this case we say that the functions $\sigma_1,\dots,\sigma_n$ are \textit{functionally independent} at the point $x\in M$. If the above condition is satisfied almost everywhere, then $\sigma_1,\dots,\sigma_n$ are called \textit{functionally independent.}

\subsection{Left-symmetric algebras}

Let $\mathfrak{a}$ be an algebra of dimension $n$ over $\mathbb{R}$ with the multiplication $\star$, i.e. a vector space with a bilinear product called $\star$. The associator $\mathcal{A}$ is a trilinear operation on $\mathfrak{a}$, defined on arbitrary triple $\xi,\eta,\zeta\in\mathfrak{a}$ as $\mathcal{A}(\xi,\eta,\zeta) = (\xi\star\eta)\star\zeta-\xi\star(\eta\star\zeta)$ for $\xi,\eta,\zeta\in\mathfrak{a}$. Notice that associators identically vanish if and only if $\mathfrak{a}$ is associative.

\begin{mydef}
	An algebra $\mathfrak{a}$ is called \textit{left-symmetric (or LSA)} if its associator is symmetric in the first two terms, i.e. 
	\begin{center}$\mathcal{A}(\xi,\eta,\zeta)=\mathcal{A}(\eta,\xi,\zeta),  \text{for all }\xi,\eta,\zeta\in\mathfrak{a}.$\end{center}
\end{mydef}

This means that left-symmetric algebras are in general non-associative.

LSAs appear in the literature with other names, namely pre-Lie algebras and Vinberg-Kozul algebras. These algebras were introduced by Vinberg \cite{EBV} in his study of the homogeneous cones. They appeared then in different frameworks of geometry, integrable systems and quantum mechanics (see also \cite{Bur}). In particular, the Novikov algebras, that play an important role in the theory of Poisson brackets of the hydrodynamic type are left-symmetric. 

\subsection{Relation between LSA and Nijenhuis operators}

We follow the description given by A. Konyaev \cite{NijGeo2}. \newline First, assume that $\mathfrak{a}$ is an arbitrary finite dimensional algebra over $\mathbb{R}$ (of dimension $n$). Then $\mathfrak{a}$ has a natural structure of a smooth $n$-dimensional affine manifold, which can be denoted as the algebra itself. Let $\eta$ be a point of this manifold, then one can naturally identify the tangent space at $\eta$ with $\mathfrak{a}$ itself. One can now define a tensor field $R$ of type (1,1) as follows: $R$ acts on element $\xi\in T_{\eta}\mathfrak{a}$ by the right-adjoint action, i.e. $R_{\eta}\xi=\xi\star\eta$. \newline We now start from the right-adjoint action defined above and we see what kind of operator we reach. \newline Fix a basis $\eta_i$ in the algebra $\mathfrak{a}$ and denote by $a_{ij}^k$ the structure constants of $\mathfrak{a}$. This basis defines a natural coordinate system on $\mathfrak{a}$, we denote it with coordinates $x^i$, i.e. $\eta=x^i\eta_i$. The components of $R_\eta$ are written as $(R_\eta)_i^k=a_{ij}^kx^j$. In particular, if we see $R_\eta$ as a linear operator, the entries of its matrix are homogeneous linear functions of $x^i$. Thus, the right-adjoint action of $\mathfrak{a}$ on itself induces the operator field with homogeneous linear components.

\begin{mydef}
	We call the above operator field $R_\eta$ the \textit{right-adjoint operator of} $\mathfrak{a}$.
\end{mydef}

Let $R_\eta$ be an operator field on a real affine space $\mathfrak{a}$ with given coordinates $x^i$ and we assume that its entries are homogeneous linear functions (we call such operator fields \textit{linear operator fields}). Then $\mathfrak{a}$ has a natural structure of algebra over $\mathbb{R}$ and its structure constants are $a_{ij}^k=\frac{\partial R_i^k}{\partial x^j}$. In this construction $R$ becomes the right-adjoint operator of $\mathfrak{a}$. This yields a natural bijection between real algebras and linear operator fields on real affine spaces.

The following proposition (\cite{NijGeo2}, Proposition 1.1) establishes the relation between homogeneous Nijenhuis operators of degree one and left-symmetric algebras.

\begin{myprop}
	Let $\mathfrak{a}$ be an algebra over $\mathbb{R}$ of dimension $n$. The following conditions are equivalent:
	\begin{enumerate}
		\item $\mathfrak{a}$ is a left-symmetric algebra
		\item The right-adjoint operator $R_{\eta}$ of $\mathfrak{a}$ is a Nijenhuis operator.
	\end{enumerate}
\end{myprop}

To learn more about this proposition, see \cite{WINT}, Theorem 1.3, Remark 1.3.

We introduce the following definition.
\begin{mydef}
	Let $\mathfrak{a}$ be a LSA and $R_{\eta}$ its right-adjoint operator. If the coefficients of the characteristic polynomial of $R_{\eta}$ are algebraically independent, then we call the algebra $\mathfrak{a}$ \textit{differentially non-degenerate}.
\end{mydef}

Clearly, algebraic independence of the coefficients of the characteristic polynomial of $R_{\eta}$ implies that almost all points in $\mathfrak{a}$ are differentially non-degenerate.

\subsection{Main Result}
The research topic of this work is one of the open problems listed in \cite{OpenProb} by Bolsinov A., Matveev V.S., Miranda E. and Tabachnikov S., namely Problem 5.16, which suggests the classification of the following equivalent objects:
\begin{enumerate}
	\item[(a)] Differentially non-degenerate LSAs
	\item[(b)] Linear Nijenhuis operator whose coefficients of the characteristic polynomial are algebraically independent. 
\end{enumerate}
A classification of these objects in dimension 2 can be easily recovered from the more general one given by A. Konyaev \cite{NijGeo2} of all 2-dimensional LSAs. In this work we give a solution to this problem in dimension 3. 
\begin{mytheo}
	The classification in dimension 3 of differentially non-degenerate LSAs or equivalently linear Nijenhuis operators whose coefficients of the characteristic polynomial are algebraically independent is presented in Table 1 and Table 2, the decomposable and indecomposable ones respectively. For each case the structure relations of the LSA, the corresponding right-adjoint Nijenhuis operator and the coefficients of its characteristic polynomial are given.
\begin{center}
	\captionof{table}{Decomposable 3-dimensional LSA}
	\begin{tabular}{|c|c|c|c|}
		\hline & & &\\[-2.5ex]
		Name & Structure relations & Nijenhuis operator $L$ & Coefficients of $\chi_{L}$\\
		\hline & & &\\[-1.5ex]
		$\mathfrak{b_4^+}\oplus\mathfrak{d}$ & $\begin{aligned}
			\eta_1\star\eta_1 &= 2\eta_1 \\
			\eta_1\star\eta_2 &= \eta_2 \\
			\eta_2\star\eta_2 &= -\eta_1 \\
			\eta_3\star\eta_3 &= \eta_3  \\
		\end{aligned}$ & $\left(\begin{array}{ccc}  
		2x & -y & 0\\
		y & 0 & 0\\ 
		0 & 0 & z 
	\end{array}\right)$ & $\begin{aligned}
								\sigma_1 &=  -2x-z\\
								\sigma_2 &=  y^2+2xz\\
								\sigma_3 &=  -y^2z\\
							\end{aligned}$\\
		\hline& & &\\[-1.5ex]
		$\mathfrak{b_4^-}\oplus\mathfrak{d}$ & $\begin{aligned}
			\eta_1\star\eta_1 &= 2\eta_1 \\
			\eta_1\star\eta_2 &= \eta_2 \\
			\eta_2\star\eta_2 &= \eta_1 \\
			\eta_3\star\eta_3 &=  \eta_3 \\
		\end{aligned}$ & $\left(\begin{array}{ccc} 
		2x & y & 0 \\ 
		y & 0 & 0 \\ 
		0 & 0 & z 
	\end{array}\right)$ & $\begin{aligned}
								\sigma_1 &=  -2x-z\\
								\sigma_2 &=  -y^2+2xz\\
								\sigma_3 &=  y^2z\\
							\end{aligned}$\\
		\hline& & &\\[-1.5ex]
		$\mathfrak{c_5^+}\oplus\mathfrak{d}$ & $\begin{aligned}
			\eta_1\star\eta_1 &= \eta_1 \\
			\eta_1\star\eta_2 &= \eta_2 \\
			\eta_2\star\eta_1 &= \eta_2 \\
			\eta_2\star\eta_2 &= \eta_1 \\
			\eta_3\star\eta_3 &= \eta_3 \\
		\end{aligned}$ &  $\left(\begin{array}{ccc} 
		x & y & 0 \\ 
		y & x & 0 \\ 
		0 & 0 & z 
	\end{array}\right)$ & $\begin{aligned}
								\sigma_1 &=  -2x-z\\
								\sigma_2 &=  x^2-y^2+2xz\\
								\sigma_3 &=  -x^2z+y^2z\\
							\end{aligned}$\\
		\hline & & &\\[-1.5ex]
		$\mathfrak{c_5^-}\oplus\mathfrak{d}$ & $\begin{aligned}
			\eta_1\star\eta_1 &= \eta_1 \\
			\eta_1\star\eta_2 &= \eta_2 \\
			\eta_2\star\eta_1 &= \eta_2 \\
			\eta_2\star\eta_2 &= -\eta_1 \\
			\eta_3\star\eta_3 &= \eta_3 \\
		\end{aligned}$ & $\left(\begin{array}{ccc} 
		x & -y & 0 \\
		y & x & 0 \\
		0 & 0 & z 
	\end{array}\right)$ & $\begin{aligned}
								\sigma_1 &=  -2x-z\\
								\sigma_2 &=  x^2+y^2+2xz\\
								\sigma_3 &=  -x^2z-y^2z\\
							\end{aligned}$\\
		\hline
	\end{tabular}

	\captionof{table}{Indecomposable 3-dimensional LSA}
	\begin{tabular}{|c|c|c|}
		\hline & &\\[-2.5ex]
		 Structure relations & Nijenhuis operator $L$ & Coefficients of $\chi_{L}$\\
		\hline & & \\[-1.5ex]
		 $\begin{aligned}
			\eta_1\star\eta_1 &= 2\eta_1 \\
			\eta_1\star\eta_2 &= \eta_2 \\
			\eta_1\star\eta_3 &= -\eta_1 \\
			\eta_2\star\eta_2 &= -\eta_1 \\
			\eta_2\star\eta_3 &= \eta_2 \\
			\eta_3\star\eta_1 &= -\eta_1 \\
			\eta_3\star\eta_3 &= \eta_1 + \eta_3 \\
		\end{aligned}$ & $\left(\begin{array}{ccc}
		2x-z & -y & z-x \\
		y & z & 0 \\
		0 & 0 & z
	\end{array}\right)$ & $\begin{aligned}
								\sigma_1 &= -2x-z\\
								\sigma_2 &=  y^2 + 4xz - z^2\\
								\sigma_3 &=  - y^2z - 2xz^2 + z^3\\
							\end{aligned}$\\
		\hline & &\\[-1.5ex]
		 $\begin{aligned}
			\eta_1\star\eta_1 &= 2\eta_1 \\
			\eta_1\star\eta_2 &= \eta_2 \\
			\eta_1\star\eta_3 &= -\eta_1 \\	
			\eta_2\star\eta_2 &= \eta_1 \\
			\eta_2\star\eta_3 &= \eta_2 \\
			\eta_3\star\eta_1 &= -\eta_1 \\
			\eta_3\star\eta_3 &= \eta_1 + \eta_3 \\
		\end{aligned}$ & $\left(\begin{array}{ccc}
		2x-z & y & z-x \\
		y & z & 0 \\
		0 & 0 & z
	\end{array}\right)$ & $\begin{aligned}
								\sigma_1 &= -2x-z \\
								\sigma_2 &=  -y^2 + 4xz - z^2\\
								\sigma_3 &=  y^2z - 2xz^2 + z^3\\
							\end{aligned}$\\
		\hline & &\\[-1.5ex]
		 $\begin{aligned}
			\eta_1\star\eta_1 &= \eta_1 \\
			\eta_1\star\eta_2 &= \eta_2 \\
			\eta_2\star\eta_3 &= -\eta_1 \\
			\eta_3\star\eta_2 &= -\eta_1 -\eta_2\\
			\eta_3\star\eta_3 &= \eta_3 \\
		\end{aligned}$ & $\left(\begin{array}{ccc}
		x & -z & -y \\
		y & 0 & -y \\
		0 & 0 & z
	\end{array}\right)$& $\begin{aligned}
								\sigma_1 &= -x-z \\
								\sigma_2 &= xz + yz \\
								\sigma_3 &= -yz^2 \\
						  \end{aligned}$\\
		\hline & &\\[-1.5ex]
		 $\begin{aligned}
			\eta_1\star\eta_1 &= -\eta_1 \\
			\eta_1\star\eta_2 &= -\tfrac{2}{3}\eta_2 \\
			\eta_1\star\eta_3 &= -\tfrac{1}{3}\eta_3 \\
			\eta_2\star\eta_3 &= \eta_1 \\
			\eta_3\star\eta_2 &= \eta_1 \\
			\eta_3\star\eta_3 &= \eta_2 \\
		\end{aligned}$ & $\left(\begin{array}{rrr}
		-x & z & y \\
		-\frac{2}{3} \, y & 0 & z \\
		-\frac{1}{3} \, z & 0 & 0
	\end{array}\right)$ & $\begin{aligned}
							\sigma_1 &= x \\
							\sigma_2 &= yz \\
							\sigma_3 &= \frac{1}{3}z^3 \\
						   \end{aligned}$\\
		\hline
	\end{tabular}
\end{center}	  
\end{mytheo}

The problem was independtly considered, and an equivalent version of Theorem 1.1 was  obtained in \cite{Sof}.

I thank Vladimir Matveev for suggesting this topic for my Master Thesis. I acknowledge partial support of DFG via  projects 455806247 and 529233771.  A part of results was obtained during my participation in the preworkshop  and workshop on Nijenhuis Geometry and Integrable Systems at the La Trobe University and Matrix Institute. I thank Alexey Bolsinov for useful discussions and Loughborough University for hospitality.
\newpage
\section{Proof}

\subsection{Algorithm}

The first step of the proof of Theorem 1.1 consists in an algorithm that compute Nijenhuis operators starting from the coefficients of its characteristic polynomials, we will denote them with $\sigma_i$. The idea of this algorithm comes from the following result (\cite{NijGeo},Corollary 2.1).

\begin{mycor}
	Let $\sigma_1,\dots,\sigma_n$ be the coefficients of the characteristic polynomial \begin{center}
		$\chi(t)=\det(t\cdot\text{Id}-L)=t^n+\sigma_1t^{n-1}+\sigma_2t^{n-2}+\cdots+\sigma_{n-1}t+\sigma_n$
	\end{center}
	of a Nijenhuis operator $L$. Then, in any local coordinate system $x_1,\dots,x_n$, the following matrix relation hold:
	\begin{equation}
		J(x)L(x)=S_{\chi}(x)J(x), \text{   where } S_{\chi}(x)=\left(
		\begin{array}{ccccc}
			-\sigma_1(x) & 1 &  &  & \\
			-\sigma_2(x)& 0 & 1 &  &  \\
			\vdots & \vdots & \ddots & \ddots &  \\
			-\sigma_{n-1}(x)& 0 & \dots & 0 & 1 \\
			-\sigma_n(x)& 0 & \dots & 0 & 0 
		\end{array}
		\right)
	\end{equation}
	and $J(x)$ is the Jacobi matrix of the collection of functions $\sigma_1,\dots,\sigma_n$ w.r.t. the variables $x_1,\dots,x_n$.
\end{mycor} 

If one assumes that d$\sigma_1,\dots,$d$\sigma_n$ are linearly independent almost everywhere, i.e. $\sigma_1,\dots,\sigma_n$ are functionally independent, then the matrix $J(x)^{-1}$ is well defined almost everywhere, and from formula (1) one can obtain: \begin{equation}L(x)=J(x)^{-1}S_{\chi}(x)J(x).
\end{equation} 

\begin{myrem}
	The operators $L$ calculated with this formula might have singular points, since inverting the matrix $J(x)$, i.e. the Jacobian, also implies dividing by the determinant of $J(x)$ and this can be zero at some points, namely where the differentials of the $\sigma$s start to be linearly dependent. This means that $L$ can have poles. Actually, if $\sigma_1,\dots,\sigma_n$ are chosen arbitrarily, in almost all the cases in which the differentials of $\sigma_1,\ldots, \sigma_n$ are linearly dependent, the formula for $L$ will have a pole at the point $x = 0$.
\end{myrem}

The general scheme of the algorithm is as follows:
\begin{itemize}
	\item \textbf{Step 1}: Choose polynomials $\sigma_i(x)$ for $i=1,\dots,n$ in a proper way with real coefficients to be determined;
	\item \textbf{Step 2}: Compute $L$ using formula (2);
	\item \textbf{Step 3}: Find the values of the coefficients of $\sigma_i$s for which the entries of $L$ become linear, in the sense that whenever a fraction in the expression of $L$ appears, then the values found for the coefficients of $\sigma_i$s are such that each denominator divides the corresponding numerator.
\end{itemize}

To  understand better the steps of the algorithm we illustrate them first for the 2-dimensional case, and then for the 3-dimensional one.

\subsubsection{2-dimensional case}

\subsubsection*{Step 1}
The polynomials $\sigma_i$ have to be chosen such that $\sigma_i(x_1,\dots,x_n)$ is homogeneous in $x_1,\dots,x_n$ of degree $i$ and such that $\sigma_1,\dots,\sigma_n$ are algebraically independent.

To avoid redundant degrees of freedom in the results, i.e. to avoid to get operators which differ only by a linear coordinate change, we simplify the polynomials $\sigma$s via linear coordinate changes, reducing the number of parameters needed. This reduction will also simplify the computations.

First we present the reduction of $\sigma_1(x_1,x_2)$. Its most general form is $\sigma_1(x_1,x_2) = ax_1+bx_2$ for $a,b\in\mathbb{R}.$ 

\begin{myprop} 
	Let $(x_1,x_2)$ and $(y_1,y_2)$ be two coordinate systems in $\mathbb{R}^2$. One can always transform \begin{center}
		\begin{tabular}{ccccc}
			$\sigma_1(x_1,x_2) = ax_1+bx_2$ & & to & & $\sigma_1(y_1,y_2)=y_1$
		\end{tabular}
	\end{center} via a linear coordinate change. Here we assume that $a,b\in\mathbb{R}$ can't be simultaneously zero, to assure $\sigma_1$ to have degree 1.
\end{myprop}

\begin{proof}
	Choose $y_1:=ax_1+bx_2$ and choose $y_2$ such that the determinant of the coordinate change matrix is not zero.
\end{proof}

We proceed now with the reduction of $\sigma_2(x_1,x_2)$.
\begin{myprop} 
	Let $(x_1,x_2)$ and $(y_1,y_2)$ be two coordinate systems in $\mathbb{R}^2$. Consider $\sigma_1(x_1,x_2)=x_1$ and 				$$\sigma_2(x_1,x_2)=a_{11}x_1^2+2a_{12}x_1x_2+a_{22}x_2^2$$
	for $a_{11},a_{12},a_{22}\in\mathbb{R}$ non simultaneously zero. Assuming that $\sigma_1$ and $\sigma_2$ have to remain algebraically independent, one can transform $\sigma_2$ to one of the following:
	\begin{align*}
		1.\quad\sigma_2(y_1,y_2) &= ay_1^2\pm y_2^2 \\
		2.\quad\sigma_2(y_1,y_2) &= y_1y_2
	\end{align*} 
	for $a\in\mathbb{R}$, via a linear coordinate change which maps $x_1$ to $y_1$.
\end{myprop}

\begin{proof}
	First observe that $a_{12}$ and $a_{22}$ can't vanish simultaneously, otherwise $\sigma_1$ and $\sigma_2$ become algebraically dependent. \newline If $a_{22}\ne0$, consider the coordinate change
	\begin{center}
		$\left(\begin{array}{c}
			x_1 \\
			x_2 
		\end{array}\right)=
		\left(\begin{array}{cc}
			1 & 0 \\
			-\frac{a_{12}}{a_{22}} & \frac{1}{\sqrt{|a_{22}|}} 
		\end{array}\right)
		\left(\begin{array}{c}
			y_1 \\
			y_2 
		\end{array}\right).$
	\end{center}
	Then, $\sigma_2(y_1,y_2)=(a_{11}-\frac{a_{12}^2}{a_{22}})y_1^2\pm y_2^2$ and setting $a:=a_{11}-\frac{a_{12}^2}{a_{22}}$ one get the case 1.
	If $a_{22}=0$, then $a_{12}\ne0$ and $a_{11}x_1^2+2a_{12}x_1x_2=x_1(a_{11}x_1+2a_{12}x_2)$ and defining $y_2:=a_{11}x_1+2a_{12}x_2$, we get $\sigma_2(y_1,y_2)=y_1y_2$, which is the second case.
\end{proof}

\subsubsection*{Step 2}

We compute manually the operators $L$ with formula (2) using the polynomials $\sigma_1,\sigma_2$ obtained in Step 1.

Let first $\sigma_1(x_1,x_2)=x_1$ and $\sigma_2(x_1,x_2)=x_1x_2$. Then
\begin{center}
	$L(x)=\left(\begin{array}{cc}
		1 & 0 \\
		-\frac{x_2}{x_1} & \frac{1}{x_1} 
	\end{array}\right)\left(\begin{array}{cc}
	-x_1  & 1 	\\
	-x_1 x_2  & 0 
\end{array}\right)\left(\begin{array}{cc}
		1 & 0 \\
		x_2 & x_1 
	\end{array}\right)=\left(\begin{array}{cc}
			-x_1 +x_2 & x_1 \\
			-\frac{x_2^2}{x_1} & -x_2 
		\end{array}\right)$
\end{center}
which clearly can't be linear and we exclude this case.

Now let $\sigma_1(x_1,x_2)=x_1$ and $\sigma_2(x_1,x_2) = ax_1^2\pm x_2^2$ for $a\in\mathbb{R}$. Then

	$$L(x)=\left(\begin{array}{cc}
		1 & 0 \\
		\mp a\frac{x_1}{x_2} & \pm\frac{1}{2x_2} 
	\end{array}\right)\left(\begin{array}{cc}
		-x_1  & 1 	\\
		-ax_1^2 \mp x_2^2  & 0 
	\end{array}\right)\left(\begin{array}{cc}
		1 & 0 \\
		2ax_1 & \pm 2x_2 
	\end{array}\right)=\left(\begin{array}{cc}
		(2a-1)x_1 & \pm 2x_2 \\
		\mp \frac{4a^2-a}{2x_2}x_1^2 -\frac{x_2}{2} & -2ax_1 
	\end{array}\right)$$
which can be linear for some values of a.

\subsubsection*{Step 3}
The operator $L$ found in the previous step is linear if and only if $4a^2-a=0$, which has solutions $a=0$ and $a=\frac{1}{4}$.
Substituting these values in the formula above we get 4 operators (2 from $\sigma_2(x_1,x_2)=ax_1^2+x_2^2$ and 2 from $\sigma_2(x_1,x_2)=ax_1^2-x_2^2$) which we report in Table 3, after having performed the coordinate change $\{x_1=-2x, x_2=y\}$. We also give their representation as LSA.
\newpage
\begin{center}
		\captionof{table}{2-dimensional differentially non-degenerate LSA}
	\begin{tabular}{|c|c|c|c|}
		\hline & & &\\[-2.5ex]
		Name & Structure relations & Nijenhuis operator $L$ & Coefficients of $\chi_{L}$\\
		\hline & & &\\[-1.5ex]
		$\mathfrak{b_4^+}$ & $\begin{aligned}
			\eta_1\star\eta_1 &= 2\eta_1 \\
			\eta_1\star\eta_2 &= \eta_2 \\
			\eta_2\star\eta_2 &= -\eta_1 \\
		\end{aligned}$ & $\left(\begin{array}{cc}  
			2x & -y \\
			y & 0  
		\end{array}\right)$ & $\begin{aligned}
			\sigma_1 &=  -2x\\
			\sigma_2 &=  y^2\\
		\end{aligned}$\\
		\hline& & &\\[-1.5ex]
		$\mathfrak{b_4^-}$ & $\begin{aligned}
			\eta_1\star\eta_1 &= 2\eta_1 \\
			\eta_1\star\eta_2 &= \eta_2 \\
			\eta_2\star\eta_2 &= \eta_1 \\
		\end{aligned}$ & $\left(\begin{array}{cc} 
			2x & y \\ 
			y & 0 
		\end{array}\right)$ & $\begin{aligned}
			\sigma_1 &=  -2x\\
			\sigma_2 &=  -y^2\\
		\end{aligned}$\\
		\hline& & &\\[-1.5ex]
		$\mathfrak{c_5^+}$ & $\begin{aligned}
			\eta_1\star\eta_1 &= \eta_1 \\
			\eta_1\star\eta_2 &= \eta_2 \\
			\eta_2\star\eta_1 &= \eta_2 \\
			\eta_2\star\eta_2 &= \eta_1 \\
			
		\end{aligned}$ &  $\left(\begin{array}{cc} 
			x & y \\ 
			y & x 
		\end{array}\right)$ & $\begin{aligned}
			\sigma_1 &=  -2x\\
			\sigma_2 &=  x^2-y^2\\
		\end{aligned}$\\
		\hline & & &\\[-1.5ex]
		$\mathfrak{c_5^-}$ & $\begin{aligned}
			\eta_1\star\eta_1 &= \eta_1 \\
			\eta_1\star\eta_2 &= \eta_2 \\
			\eta_2\star\eta_1 &= \eta_2 \\
			\eta_2\star\eta_2 &= -\eta_1 \\
		\end{aligned}$ & $\left(\begin{array}{cc} 
			x & -y \\
			y & x
		\end{array}\right)$ & $\begin{aligned}
			\sigma_1 &=  -2x\\
			\sigma_2 &=  x^2+y^2\\
		\end{aligned}$\\
		\hline
	\end{tabular}
\end{center}
Table 3 gives the classification of 2-dimensional differentially non-degenerate LSAs and agrees with the one given by A. Konyaev \cite{NijGeo2}.

\subsubsection{3-dimensional case}

\subsubsection*{Step 1}

We present the reduction of $\sigma_1$ and $\sigma_2$ in dimension 3.

\begin{myprop}
		Let $(x_1,x_2,x_3)$ and $(y_1,y_2,y_3)$ be two coordinate systems of $\mathbb{R}^3$. Given\begin{align*}
		\sigma_1(x_1,x_2,x_3) &= a x_1 + b x_2 + c x_3 \\
		\sigma_2(x_1,x_2,x_3) &= a_{11} x_1^{2} + 2 \, a_{12} x_1 x_2 + a_{22} x_2^{2} + 2 \, a_{13} x_1 x_3 + 2 \, a_{23} x_2 x_3 + a_{33} x_3^{2}
	\end{align*}
	where $a,b,c,a_{11},a_{12},a_{22},a_{13},a_{23},a_{33}\in\mathbb{R}$ and such that $a,b,c$ and also $a_{11},a_{12},a_{22},a_{13},a_{23},a_{33}$ are not simultaneously zero, one can always find a linear coordinate change such that $$\sigma_1(y_1,y_2,y_3) = y_1 $$ and $\sigma_2$ is one of
	\begin{align*}
		1.\quad\sigma_2(y_1,y_2,y_3) &= \alpha y_1^2 \pm y_2^2 \pm y_3^2 \\
		2.\quad\sigma_2(y_1,y_2,y_3) &= \alpha y_1^2 \pm y_2^2 \\
		3.\quad\sigma_2(y_1,y_2,y_3) &= y_1y_2\\ 
		4.\quad\sigma_2(y_1,y_2,y_3) &= y_1y_2 \pm y_3^2\\  
		5.\quad\sigma_2(y_1,y_2,y_3) &= \alpha y_1^2
	\end{align*} 
	for $\alpha\in\mathbb{R}$.
\end{myprop}

\begin{proof}
W.l.o.g. assume $a\ne0$ and take $a=1$. Consider 
	\begin{center}$\left(\begin{array}{c}
			x_1 \\
			x_2 \\
			x_3
		\end{array}\right)=
		\left(\begin{array}{ccc}
			1-b\delta-c\eta & -b\epsilon-c\theta & -b\zeta-c\iota \\
			\delta & \epsilon & \zeta \\
			\eta & \theta & \iota
		\end{array}\right)
		\left(\begin{array}{c}
			y_1 \\
			y_2 \\
			y_3
		\end{array}\right)$\end{center}
	 where $\delta,\epsilon,\zeta,\eta,\theta,\iota \in\mathbb{R}$ and such that the determinant of matrix $\epsilon\iota-\zeta\theta$ is not zero. This coordinate change is such that $\sigma_1(y_1,y_2,y_3)=y_1$ as we wanted. Thus, without loss of generality we may assume that $\sigma_1(x_1,x_2,x_3)=x_1$.
	 
	 Now we reduce $\sigma_2$, finding a coordinate change from $(x_1,x_2,x_3)$ to $(y_1,y_2,y_3)$ which fixes $y_1=x_1$ to preserve the shape of $\sigma_1$. First let us assume $a_{33}\ne0$.
	 Then 
	 \begin{equation*}
	 	\sigma_2(x_1,x_2,x_3) = a_{33}\left(\frac{a_{11}}{a_{33}} x_1^{2} + 2 \frac{a_{12}}{a_{33}} x_1 x_2 + \frac{a_{22}}{a_{33}} x_2^{2} + 2 \left(\frac{a_{13}}{a_{33}} x_1 + \frac{a_{23}}{a_{33}} x_2\right) x_3 + x_3^{2}\right)
	 \end{equation*}
	 and adding and subtracting $\left(\frac{a_{13}}{a_{33}} x_1 + \frac{a_{23}}{a_{33}} x_2\right)^2$ we get:
	\begin{equation}
		\sigma_2(x_1,x_2,y_3) = \left(a_{11} - \frac{a_{13}^2}{a_{33}}\right)x_1^{2} + 2\left(a_{12}-\frac{a_{13}a_{23}}{a_{33}}\right) x_1 x_2 + \left(a_{22}-\frac{a_{23}^2}{a_{33}}\right) x_2^{2} \pm y_3^2
	\end{equation}
	 where $y_3:= \sqrt{|a_{33}|}\left(\frac{a_{13}}{a_{33}} x_1 + \frac{a_{23}}{a_{33}} x_2 +x_3 \right)$. We now proceed analogously with the coefficient of $x_2^2$. We first assume $\left(a_{22}-\frac{a_{23}^2}{a_{33}}\right)\ne0$. Then we divide and multiply again every term (except $y_3^2$) with it and proceeding in the same way as above we get
	 \begin{equation*}
	 \sigma_2(y_1,y_2,y_3) = \alpha y_1^2\pm y_2^2\pm y_3^2
	 \end{equation*}
	 where $y_2:= \sqrt{|a_{22}-\frac{a_{23}^2}{a_{33}}|}\left(\frac{a_{12}a_{33}-a_{13}a_{23}}{a_{22}a_{33}-a_{23}^2} x_1 + x_2\right)$ and $\alpha = \frac{a_{11}a_{22}a_{33}-a_{11}a_{23}^2-a_{13}^2a_{22}-a_{12}^2a_{33}+2-a_{12}a_{13}a_{23}}{a_{22}a_{33}-a_{23}^2}$, and hence we arrive to Case 1. \newline Now we assume again $a_{33}\ne0$ but $\left(a_{22}-\frac{a_{23}^2}{a_{33}}\right)=0$. Then equation (3) becomes
	 \begin{equation*}
	 	\sigma_2(x_1,x_2,y_3) = \left(a_{11} - \frac{a_{13}^2}{a_{33}}\right)x_1^{2} + 2\left(a_{12}-\frac{a_{13}a_{23}}{a_{33}}\right) x_1 x_2 \pm y_3^2
	 \end{equation*}
	 and collecting $x_1$ in the first two terms and setting $y_2:=\left(a_{11} - \frac{a_{13}^2}{a_{33}}\right)x_1 + 2\left(a_{12}-\frac{a_{13}a_{23}}{a_{33}}\right) x_2$ we get the Case 4, i.e. $\sigma_2(y_1,y_2,y_3)=y_1y_2\pm y_3^2$.
	 
	 Now let us assume $a_{33}=0$ and $a_{22}\ne0$. To be in a different situation from the ones above, assume $a_{13}=0$. Then, we repeat the same trick as above and we get the Case 2, i.e. $\sigma_2(y_1,y_2,y_3) = \alpha y_1^2 \pm y_2^2$. 
	 
	 Now let $a_{33}=a_{22}=0$. If also $a_{23}=0$, then 
	 \begin{equation*}
		\sigma_2(x_1,x_2,x_3) = a_{11} x_1^{2} + 2 a_{12} x_1 x_2 + 2 a_{13} x_1 x_3 = x_1(a_{11} x_1 + 2 a_{12} x_2 + 2 a_{13} x_3)
	 \end{equation*} 
	 and defining $y_2:=a_{11} x_1 + 2 a_{12} x_2 + 2 a_{13} x_3$ we get $\sigma_2(y_1,y_2,y_3)=y_1y_2$, i.e. Case 3.
	 
	 Finally let us assume $a_{33}=a_{22}=a_{12}=a_{13}=a_{23}=0$. Then setting $\alpha:=a_{11}$, one get Case 5 and the proof is completed. 	 
\end{proof}

\begin{myrem}
	The $\sigma_2$ of Case 5 is algebraically dependent with $\sigma_1$, hence we will not consider it. We will see that only Case 1 gives us interesting results. 
\end{myrem}

Proposition 2.3 gives us nicer shape for $\sigma_1$ and $\sigma_2$. Concerning $\sigma_3$, we will use it in its most general form, i.e. 
\begin{equation*}
\begin{split}
	\sigma_3(x_1,x_2,x_3) = b_{11} x_1^{3} + 3b_{12} x_1^{2} x_2 + 3b_{21} x_1 x_2^{2} + b_{22} x_2^{3} + 3b_{13} x_1^{2} x_3\\
	 + 6 c x_1 x_2 x_3  + 3 b_{23} x_2^{2} x_3	+ 3 b_{31} x_1 x_3^{2} + 3 b_{32} x_2 x_3^{2} + b_{33} x_3^{3} \
\end{split}
\end{equation*}
for $c,b_{ij}\in\mathbb{R}$ for $i,j=1,2,3$.

\subsubsection*{Step 2}

We compute, with a computer algebra program, the operator $L$ via the formula (2), using the $\sigma$s found in the previous paragraph. In this way we find an operator $L(x)$ whose first row is linear, due to the form of the initial $\sigma$s. Indeed, the first row of the resulting $J^{-1}$ is $(1,0,0)$ and this means that the first row of $L$ is the first row of the product $S_{\chi}J$, which is $(-x_1+\partial_{x_1}\sigma_2,\partial_{x_2}\sigma_2, \partial_{x_3}\sigma_2)$. The two remain rows are not linear, but one can observe that they have all the same denominator. Namely, $L$ has the following form:
\begin{center}
	$\left(\begin{array}{ccc}
		\dots & \dots & \dots \\
		\cfrac{P_1(x)}{Q(x)} & \cfrac{P_2(x)}{Q(x)} & \cfrac{P_3(x)}{Q(x)} \\
		\cfrac{P_4(x)}{Q(x)} & \cfrac{P_5(x)}{Q(x)} & \cfrac{P_6(x)}{Q(x)} 
	\end{array}\right)$
\end{center}
where $P_i(x),Q(x)$ for $i=1,\dots,6$ are homogeneous polynomials in $x$, of degree 4 and 3 respectively, and depend on the parameters that compare in the $\sigma$s. The denominator $Q$ moreover coincides with the determinant of the Jacobian matrix. 

\subsubsection*{Step 3}
 We want to find values for the coefficients of $\sigma$s such that all terms of $L$ become linear, i.e. such that each numerator $P_i$ is divisible by $Q$. To achieve this goal, we consider for each $i=1,\dots,6$ the following equation:
\begin{equation}
	P_i(x)-Q(x)(\alpha_{i1}x_1+\alpha_{i2}x_2+\alpha_{i3}x_3)=0
\end{equation} 
where $\alpha_{i,j}$ for $i=1,\dots,6$ and $j=1,2,3$ are additional parameters in $\mathbb{R}$. If these equations are identically satisfied, then our entries of $L$ are linear. 

The equations (4) are identically satisfied if and only if for all $i$ the coefficients of the polynomial in its LHS vanish. Hence, the strategy is now to solve the system of equations formed by the coefficients of (4). The sageMath code that I used to generate this system of equations is reported in Listing 1.

Now we study each case from 1 to 4 separately. We start with the easiest ones.

\paragraph{Case 3} Let $\sigma_1(x_1,x_2,x_3)=x_1$ and $\sigma_2(x_1,x_2,x_3)=x_1x_2$. In this case, one can observe quite directly that it is not possible to make all the entries of $L$ linear in coordinates. We compute $L$ with formula (2) and we see that most of the non linear terms are simpler than in the general case, namely, $P_i$ and ${Q}$ have some common divisor for $i=1,2,3,4$ and the corresponding entries of $L$ become fractions with a polynomial of degree 2 at the numerator and a monomial of degree 1 at the denominator, namely the denominator is $kx_1$ for some $k\in\mathbb{R}$. Hence, a necessary condition for making these entries linear in $x$, is that the coefficients of the terms of the numerator in which $x_1$ does not appear, have to vanish, which in this case means:
\begin{equation*}
	b_{21}=\frac{1}{3}, c=b_{31}=b_{22}=b_{23}=b_{32}=b_{33}=0.
\end{equation*}
Substituting these values and computing again $L$ one sees that $\frac{P_4}{Q}$ and $\frac{P_5}{Q}$ can't become linear. For example the term $\frac{P_5}{Q}$ is: 
\begin{equation*}
-\frac{(3b_{12}^2+b_{11})x_1^2+5b_{12}x_1x_2+x_2^2+2b_{13}x_1x_3}{b_{13}x_1}
\end{equation*}
which, due to $x_2^2$ in the numerator, can't be linear. 

\paragraph{Case 2} Now let $\sigma_2(x_1,x_2,x_3)=ax_1^2\pm x_2^2$ for some $a\in\mathbb{R}$. As in Case 3, one can see quite directly that not all the entries of $L$ can be linear. Also in this case,  $P_i$ and ${Q}$ have some common divisor for $i=1,2,3,4$ and the corresponding entries of $L$ become fractions with a polynomial of degree 2 at the numerator and a monomial of degree 1 at the denominator, this time the denominator is of the form $kx_2$ for some $k\in\mathbb{R}$. Proceeding as in Case 3, we find conditions for the parameters:
\begin{equation*}
	b_{11}=\frac{4a^2-a}{3}, c=b_{31}=b_{12}=b_{13}=b_{32}=b_{33}=0.
\end{equation*}
Substituting these values and computing again $L$, one sees that a necessary condition for $\frac{P_5}{Q}$ to be linear is $b_{23}=0$, which makes the determinant of $J$ zero and hence it's not acceptable. 

The remaining cases are more complicated, due to the presence of all three variables $x_1,x_2,x_3$ in the expressions for $\sigma_2$, and we will actually need to solve the system of equation (4) we created. 

\paragraph{Case 4} Let $\sigma_2(x_1,x_2,x_3)=x_1x_2\pm x_3^2$. We generate the system of equations (4). Looking at the system one observes that to have solutions $b_{22}$ and $b_{23}$ have to vanish. After substituting these values, one equation of the system become $6b_{21}^2-b_{21}=0$. We proceed then substituting the two solutions for $b_{21}$ in the system. For $b_{21}=0$ one gets that also $b_{32},c,b_{13},b_{33},b_{31},b_{12}$ have to be zero, for which then only the parameter $b_{11}$ survives. But then $\sigma_3$ is functionally dependent with $\sigma_1$, hence it is not an interesting solution for us. Instead for $b_{21}=\frac{1}{6}$, one gets that $b_{32}=\pm \frac{1}{3}$ and $b_{32}=0$, which is a contradiction. Hence, also Case 4 gives us no interesting solutions.

\paragraph{Case 1} Now we deal with $\sigma_2(x_1,x_2,x_3)=ax_1^2\pm x_2^2\pm x_3^2$.
First we remark that if the signs of $x_2^2$ and $x_3^2$ are different, then with the following coordinate change 
\begin{gather*}
	y_1=x_1\\
	y_2 = x_2+x_3\\
	y_3 = x_2-x_3,
\end{gather*}
one gets $\sigma_2(y_1,y_2,y_3)=ay_1^2\pm y_2y_3$. We use this new shape for $\sigma_2$, since the computations become easier. \newline Computing the system of equations (4), ones sees that $b_{12}$ and $b_{13}$ have to be zero. We compute again the system setting $b_{12}=b_{13}=0$ and solve it with Maple. It gave us the solutions listed in Listing 2. Analyzing them (Listing 3), we find 5 different non trivial solutions which give us the operators $L_1,L_2,L_3,L_4,L_5$ presented in Table 4.

Now let $\sigma_2(x_1,x_2,x_3)=ax_1^2 - x_2^2 - x_3^2$. Again from the computed system of equation (4) one notice that $b_{12}$ and $b_{13}$ have to be zero in order to have solutions. This time Maple didn't give us a solution in a reasonable amount of time and hence we proceed in an other way. \newline We notice that $\sigma_1$ and $\sigma_2$ are rotation invariant, namely they remain constant after performing the transformation:  
\begin{align*}
	y_1 &= x_1\\
	y_2 &= x_2  \cos{\phi} + x_3 \sin{\phi}\\
	y_3 &= - x_2 \sin{\phi} + x_3 \cos{\phi}
\end{align*}
For this reason one can make one coefficient of $\sigma_3$ vanish. Hence I set each coefficient to zero, one at a time, and consider the corresponding system of equations. I collect all the solutions found and select only the non equivalent and non trivial ones. These are presented in Listing 4. The three resulting operators are again reported in Table 4, namely operator $L_6,L_7$ and $L_8$. \newline We remark that setting $c=b_{22}=b_{32}=0$ (or equivalently $c=b_{33}=b_{23}=0$) one still finds all the non equivalent solutions.

Now let us consider the remaining case: $\sigma_2(x_1,x_2,x_3)=ax_1^2 + x_2^2 + x_3^2$. Here we proceed exactly the same as for the previous case. But, we only found trivial or complex solutions, and hence no interesting solutions for us.
\begin{center}
	\captionof{table}{3-dimensional classification of linear Nijenhuis operators with algebraically independent coefficients of their characteristic polynomials}
	\begin{tabular}{|ccc|}
		\hline
		\multirow{2}{100pt}{\minitab[c]{$\sigma_1=x_1$}} & \multirow{2}{100pt}{\minitab[c]{$\sigma_2=x_2x_3$}} & \multirow{2}{200pt}{\minitab[c]{$\sigma_3=\frac{1}{3}x_3^3$}} \\
		& & \\
		\hline
		\multicolumn{3}{|c|}{} \\
		\multicolumn{3}{|c|}{$L_1:=\left(\begin{array}{rrr}
				-x_{1} & x_{3} & x_{2} \\
				-\frac{2}{3} \, x_{2} & 0 & x_{3} \\
				-\frac{1}{3} \, x_{3} & 0 & 0
			\end{array}\right)$} \\
		\multicolumn{3}{|c|}{} \\
		\hline
		\hline
		\multirow{2}{100pt}{\minitab[c]{$\sigma_1=x_1$}} & \multirow{2}{100pt}{\minitab[c]{$\sigma_2=x_2 x_3$}} & \multirow{2}{200pt} {\minitab[c]{$\sigma_3=-x_1x_2^2 - x_2^3 - x_2^2x_3$}} \\
		& & \\
		\hline
		\multicolumn{3}{|c|}{} \\
		\multicolumn{3}{|c|}{$L_2:=\left(\begin{array}{rrr}
				-x_{1} & x_{3} & x_{2} \\
				0 & x_{2} & 0 \\
				-x_{2} - x_{3} & -2 \, x_{1} - 3 \, x_{2} - 3 \, x_{3} & -x_{2}
			\end{array}\right)$} \\
		\multicolumn{3}{|c|}{} \\
		\hline
		\hline
		\multirow{2}{100pt}{\minitab[c]{$\sigma_1=x_1$}} & \multirow{2}{100pt}{\minitab[c]{$\sigma_2=\frac{1}{3}x_1^2 + x_2x_3$}} & \multirow{2}{200pt}{$\sigma_3=\frac{1}{27}x_1^3-\frac{1}{216}x_2^3+\frac{1}{3}x_1x_2x_3 -\frac{1}{8}x_2^2x_3+\frac{3}{8}x_2x_3^2+\frac{1}{8}x_3^3$} \\
		& & \\
		\hline
		\multicolumn{3}{|c|}{} \\
		\multicolumn{3}{|c|}{$L_3:=\left(\begin{array}{rrr}
				-\frac{1}{3} \, x_{1} & x_{3} & x_{2} \\
				-\frac{1}{3} \, x_{2} & -\frac{1}{3} \, x_{1} + \frac{1}{8} \, x_{2} + \frac{3}{8} \, x_{3} & \frac{9}{8} \, x_{2} + \frac{3}{8} \, x_{3} \\
				-\frac{1}{3} \, x_{3} & -\frac{1}{72} \, x_{2} - \frac{3}{8} \, x_{3} & -\frac{1}{3} \, x_{1} - \frac{1}{8} \, x_{2} - \frac{3}{8} \, x_{3}
			\end{array}\right)$} \\
		\multicolumn{3}{|c|}{} \\
		\hline
		\hline
		\multirow{2}{100pt}{\minitab[c]{$\sigma_1=x_1$}} & \multirow{2}{100pt}{\minitab[c]{$\sigma_2=\frac{1}{3}x_1^2 + x_2x_3$}} & \multirow{2}{200pt}{\minitab[c]{$\sigma_3=\frac{1}{27}x_1^3 -\frac{1}{27}x_2^3 + x_3^3 + \frac{1}{3}x_1x_2x_3$}} \\
		& & \\
		\hline
		\multicolumn{3}{|c|}{} \\
		\multicolumn{3}{|c|}{$L_4:=\left(\begin{array}{rrr}
				-\frac{1}{3} \, x_{1} & x_{3} & x_{2} \\
				-\frac{1}{3} \, x_{2} & -\frac{1}{3} \, x_{1} & 3 \, x_{3} \\
				-\frac{1}{3} \, x_{3} & -\frac{1}{9} \, x_{2} & -\frac{1}{3} \, x_{1}
			\end{array}\right)$} \\
		\multicolumn{3}{|c|}{} \\
		\hline
		\end{tabular}
		\begin{tabular}{|ccc|}
		\hline
		\multirow{2}{100pt}{\minitab[c]{$\sigma_1=x_1$}} & \multirow{2}{100pt}{\minitab[c]{$\sigma_2=\frac{1}{4}x_1^2 + x_2x_3$}} & \multirow{2}{200pt}{$\sigma_3=\frac{1}{2}x_1x_2^2 - x_2^3 + \frac{1}{4}x_1x_2x_3 - \frac{1}{4}x_2^2x_3 	+ \frac{1}{32}x_1x_3^2 + \frac{1}{16}x_2x_3^2 + \frac{1}{64}x_3^3$} \\
		& & \\
		\hline
		\multicolumn{3}{|c|}{} \\
		\multicolumn{3}{|c|}{$L_5:=\left(\begin{array}{rrr}
				-\frac{1}{2} \, x_{1} & x_{3} & x_{2} \\
				-\frac{3}{8} \, x_{2} + \frac{1}{32} \, x_{3} & -\frac{1}{4} \, x_{1} + \frac{1}{4} \, x_{2} + \frac{1}{16} \, x_{3} & \frac{1}{16} \, x_{1} + \frac{3}{16} \, x_{2} + \frac{3}{64} \, x_{3} \\
				\frac{1}{2} \, x_{2} - \frac{3}{8} \, x_{3} & x_{1} - 3 \, x_{2} - \frac{3}{4} \, x_{3} & -\frac{1}{4} \, x_{1} - \frac{1}{4} \, x_{2} - \frac{1}{16} \, x_{3}
			\end{array}\right)$} \\
		\multicolumn{3}{|c|}{} \\
		\hline
		\hline
		\multirow{2}{100pt}{\minitab[c]{$\sigma_1=x_1$}} & \multirow{2}{100pt}{\minitab[c]{$\sigma_2=\frac{1}{4}x_1^2 - x_2^2-x_3^2 $}} & \multirow{2}{200pt}{\minitab[c]{$\sigma_3=- \frac{1}{2} x_1 x_2^2 + x_2^2 x_3$}} \\
		& & \\
		\hline
		\multicolumn{3}{|c|}{} \\
		\multicolumn{3}{|c|}{$L_6:=\left(\begin{array}{rrr}
				-\frac{1}{2} \, x_{1} & -2 \, x_{2} & -2 \, x_{3} \\
				-\frac{1}{4} \, x_{2} & 0 & -\frac{1}{2} \, x_{2} \\
				-\frac{1}{2} \, x_{3} & -x_{2} & -\frac{1}{2} \, x_{1}
			\end{array}\right)$} \\
		\multicolumn{3}{|c|}{} \\
		\hline
		\hline
		\multirow{2}{100pt}{\minitab[c]{$\sigma_1=x_1$}} & \multirow{2}{100pt}{\minitab[c]{$\sigma_2=\frac{1}{3}x_1^2 - x_2^2-x_3^2$}} & \multirow{2}{200pt} {$\sigma_3=\frac{1}{27} x_1^3 - \frac{1}{3}x_1x_3^2 - \frac{1}{3}x_1x_2^2 - \frac{1}{\sqrt{3}}x_2^2x_3 - \frac{2}{3\sqrt{3}}x_3^3 $} \\
		& & \\
		\hline
		\multicolumn{3}{|c|}{} \\
		\multicolumn{3}{|c|}{$L_7:=\left(\begin{array}{rrr}
				-\frac{1}{3} \, x_{1} & -2 \, x_{2} & -2 \, x_{3} \\
				-\frac{1}{3} \, x_{2} & -\frac{1}{9} \, \sqrt{3} {\left(\sqrt{3} x_{1} + 3 \, x_{3}\right)} & \frac{1}{6} \, \sqrt{3} x_{2} \\
				-\frac{1}{3} \, x_{3} & \frac{2}{3} \, \sqrt{3} x_{2} & -\frac{1}{9} \, \sqrt{3} {\left(\sqrt{3} x_{1} - 3 \, x_{3}\right)}
			\end{array}\right)$} \\
		\multicolumn{3}{|c|}{} \\
		\hline
		\hline
		\multirow{2}{100pt}{\minitab[c]{$\sigma_1=x_1$}} & \multirow{2}{100pt}{\minitab[c]{$\sigma_2=\frac{1}{3}x_1^2 - x_2^2-x_3^2$}} & \multirow{2}{200pt}{$\sigma_3=\frac{1}{27}x_1^3 - \frac{1}{3}x_1x_3^2 - \frac{1}{3}x_1x_2^2 + \frac{2}{\sqrt{3}} x_2^2x_3 - \frac{2}{3\sqrt{3}}x_3^3$} \\
		& & \\
		\hline
		\multicolumn{3}{|c|}{} \\
		\multicolumn{3}{|c|}{$L_8:=\left(\begin{array}{rrr}
				-\frac{1}{3} \, x_{1} & -2 \, x_{2} & -2 \, x_{3} \\
				-\frac{1}{3} \, x_{2} & -\frac{1}{9} \, \sqrt{3} {\left(\sqrt{3} x_{1} + 3 \, x_{3}\right)} & -\frac{1}{3} \, \sqrt{3} x_{2} \\
				-\frac{1}{3} \, x_{3} & -\frac{1}{3} \, \sqrt{3} x_{2} & -\frac{1}{9} \, \sqrt{3} {\left(\sqrt{3} x_{1} - 3 \, x_{3}\right)}
			\end{array}\right)$} \\
		\multicolumn{3}{|c|}{} \\
		\hline
	\end{tabular}
\end{center}

\begin{myrem}
	Since we considered all possible combinations (up to linear coordinate changes) of algebraically independent $\sigma_1,\sigma_2,\sigma_3$ such that deg$(\sigma_i)=i$ we got all possible differentially non-degenerate LSAs.
\end{myrem}

\subsection*{Simplification}

The second part of the proof consist of writing the classification in a nicer way in which one can directly get some more information about the operators found and to derive the structure relations of the corresponding LSAs. \newline For each operator from Table 4, except $L_1$ which already has a nice form, we perform suitable coordinate change and get a new representation of them, as presented below.
\begin{itemize}
	\item[]\begin{equation*}
		L_2:
		\begin{split}
			x_1 &= -x-y \\
			x_2 &= x+y+z \\
			x_3 &= y
		\end{split}
		\qquad \qquad \qquad \qquad
		\left(\begin{array}{ccc}
			x & -z & -y \\
			y & 0 & -y \\
			0 & 0 & z
		\end{array}\right)
	\end{equation*}
	\item[]\begin{equation*}
		L_3:
		\quad
		\begin{split}
			x_1 &= -\frac{1}{3}x-\frac{1}{12}y-+\frac{1}{4}z\\ 
			x_2 &= \frac{\sqrt{3}}{6}y+\frac{\sqrt{3}}{2}z\\
			x_3 &= -\frac{1}{3}x+\frac{1}{6}y-\frac{1}{2}z
		\end{split}
		\qquad 
		\left(\begin{array}{ccc}
			2x-z & -y & z-x \\
			y & z & 0 \\
			0 & 0 & z
		\end{array}\right)
	\end{equation*}	
	\item[]\begin{equation*}
		L_4:
		\begin{split}
			x_1 &= -\frac{1}{3} \, x - \frac{1}{6} \, y + \frac{1}{2} \, z \\
			x_2 &= \frac{1}{6} \, \sqrt{3} y + \frac{1}{2} \, \sqrt{3} z \\
			x_3 &= -\frac{1}{3} \, x + \frac{1}{3} \, y - z
		\end{split}
		\qquad \qquad \qquad
		\left(\begin{array}{ccc} 
			x & -y & 0 \\ 
			y & x & 0 \\ 
			0 & 0 & z 
		\end{array}\right)
	\end{equation*}
	\item[]\begin{equation*}
		L_5:
		\begin{split}
			x_1 &= -\frac{1}{4}x-\frac{1}{2}y+\frac{1}{8}z \\
			x_2 &= \pm(y+\frac{1}{4}z) \\
			x_3 &= -\frac{1}{2}x+y-\frac{1}{4}z
		\end{split}
		\qquad \qquad \qquad
		\left(\begin{array}{ccc}  
			2x & -y & 0\\  
			y & 0 & 0\\  
			0 & 0 & z 
		\end{array}\right)
	\end{equation*}
	\item[]\begin{equation*}
		L_6:
		\begin{split}
			x_1 &= -\frac{1}{4}x-\frac{1}{2}z \\
			x_2 &= \pm y \\
			x_3 &= -\frac{1}{2}x+z
		\end{split}
		\qquad \qquad \qquad \qquad
		\left(\begin{array}{ccc} 
			2x & y & 0 \\ 
			y & 0 & 0 \\ 
			0 & 0 & z 
		\end{array}\right)
	\end{equation*}
	\item[]\begin{equation*}
		L_7:
		\quad
		\begin{split}
			x_1 &= -\frac{1}{3}x+\frac{\sqrt{3}}{6}z \\
			x_2 &= y \\
			x_3 &= -\frac{1}{3}x-\frac{\sqrt{3}}{3}z
		\end{split}
		\qquad \qquad \qquad
		\left(\begin{array}{ccc}
			2x-z & y & z-x \\
			y & z & 0 \\
			0 & 0 & z
		\end{array}\right)
	\end{equation*}
	\item[]\begin{equation*}
		L_8:
		\qquad
		\begin{split}
			x_1 =& -\frac{1}{3}y_1+y_2-\frac{\sqrt{3}}{3}y_3 \\
			x_2 =& -\frac{1}{3}y_1-y_2-\frac{\sqrt{3}}{3}y_3 \\
			x_3 =& -\frac{1}{3}y_1+\frac{2\sqrt{3}}{3}y_3
		\end{split}
		\qquad 
		\left(\begin{array}{ccc} 
			y_1 & 0 & 0 \\ 
			0 & y_2 & 0 \\ 
			0 & 0 & y_3 
		\end{array}\right)
		\simeq
		\left(\begin{array}{ccc} 
			x & y & 0 \\ 
			y & x & 0 \\ 
			0 & 0 & z 
		\end{array}\right)
	\end{equation*}		
\end{itemize}
where $(y_1,y_2,y_3)$ and $(x,y,z)$ are two other coordinate systems of $\mathbb{R}^3$ and where the last isomorphism for $L_8$ is given by $\{x=2(y_1+y_2),y=2(y_1-y_2),z=y_3\}$. \newline We then compute the structure relations as describe above, i.e. $\eta_i\star\eta_j=a_{ij}^k\eta_k$ for $i,j=1,2,3$ where $a_{ij}^k = \frac{\partial R^k_i}{\partial x_j}$ and $R_i^k$ are the components of the operators we found. This completes the proof of Theorem 1.1. 

\section{Discussion}

\begin{myrem}
The differentially non-degenerate LSAs presented in Table 1 are decomposable, in the sense that they can be written as direct product of differentially non-degenerate LSAs of lower dimension. We remark that the direct product of two (linear) Nijenhuis operators is a (linear) Nijenhuis operator and similarly that the product of two (differentially non-degenerate) LSAs is again a (differentially non-degernerate) LSA. As already written in Table 1, denoting with $\mathfrak{d}$ the only possible 1-dimensional non trivial LSA, for which the structure relation is $\eta_1\star \eta_1 = \eta_1$ and the Nijenhuis operator is the $1\times1$ matrix $L=(x)$ and using for dimension 2 the notation of Table 3, we see that $$L_5\simeq\mathfrak{b_4^+}\oplus\mathfrak{d},\quad L_6\simeq\mathfrak{b_4^-}\oplus\mathfrak{d},\quad L_8\simeq\mathfrak{c_5^+}\oplus\mathfrak{d}\simeq\mathfrak{d}\oplus\mathfrak{d}\oplus\mathfrak{d},\quad L_4\simeq\mathfrak{c_5^-}\oplus\mathfrak{d}.$$ This means that in our classification all possible 3-dimensional combinations of lower dimensional differentially non-degenerate LSAs are present. The remaining LSAs in the classification are indecomposable.
\end{myrem}

\begin{myrem}
	We worked in $\mathbb{R}$. If we look at $\mathbb{C}$, we have that  $\mathfrak{b_4^+}\simeq\mathfrak{b_4^-}$ and $\mathfrak{c_5^+}\simeq\mathfrak{c_5^-}$ over $\mathbb{C}$. As a consequence in dimension 3, $\mathfrak{b_4^+}\oplus\mathfrak{d}\simeq\mathfrak{b_4^-}\oplus\mathfrak{d}$ and $\mathfrak{c_5^+}\oplus\mathfrak{d}\simeq\mathfrak{c_5^-}\oplus\mathfrak{d}$ over $\mathbb{C}$. Moreover, also $L_3$ and $L_7$ are isomorphic over $\mathbb{C}$. This gives us a classification in dimension 3 of differentially non-degenerate LSAs over $\mathbb{C}$. Note that classify LSAs is in general not an easy task, due to their non associativity.
\end{myrem}

\begin{myrem}
	Looking at the indecomposable operators of our list, one could try to generalise them in dimension $n$. The generalisations we found are reported below.
\paragraph{$\mathbf{L_1}$:}
		 Set 
				\begin{equation*}
			\begin{split}
				&\sigma_1 = x_1 \\
				&\sigma_2 = x_2x_n \\
				&\vdots \\
				&\sigma_i = x_ix_n^{i-1} \\
				&\vdots \\
				&\sigma_{n-1} = x_{n-1}x_n^{n-2} \\
				&\sigma_n =\frac{1}{n}x_n^n \\
			\end{split}
		\end{equation*}
	(If $n=2k$, one should also consider the case $\sigma_n = -\frac{1}{n}x_n^n$).
		To get an explicit expression in $x_s$ for the Nijenhuis operator, we use the formula (2) with
		$$S_{\chi} = \left(\begin{array}{cccccc}
			-\sigma_1 & 1 &   &   &   & \\
			-\sigma_2 & 0 & 1 &   &   & \\
			-\sigma_3 &   & 0 & 1 &   & \\
			\vdots    &   &   & \ddots & \ddots & \\
			-\sigma_{n-1} &   &   &   & 0 & 1 \\ 
			-\sigma_n & 0 &   & \dots &   & 0 
		\end{array}\right).$$
		Therefore:
		$$J:= \left(\frac{\partial\sigma_i}{\partial x_j}\right) = \left(\begin{array}{cccccc}
			1 &  &   &   &   & 0\\
			& x_n &  &   &   & x_2\\
			&   & x_n^2 &  &   & 2x_3x_n\\
			&   &   & \ddots &  & \vdots\\
			&   &   &   & x_n^{n-2} & (n-2)x_{n-1}x_n^{n-3} \\ 
			&   &   &   &   & x_n^{n-1}
		\end{array}\right)$$
		and hence we get the following Nijenhuis operator:
		$$\left(\begin{array}{ccccccc}
			-x_1 & x_n &   &   &  &  & x_2\\
			\frac{1-n}{n}x_2 & 0 & x_n &   &  &  & 2x_3\\
			\frac{2-n}{n}x_3 &   & 0 & x_n &  &  & 3x_4\\
			\vdots    &   &  &  \ddots & \ddots &  & \vdots\\
			\vdots    &   &   &    &  0 & x_n & (n-2)x_{n-1}\\
			-\frac{2}{n}x_{n-1} &   &   &   & & 0 & x_n \\ 
			-\frac{1}{n}x_n & 0 &   & \dots &  &   & 0 
		\end{array}\right) $$
		From this we compute the structure relations as above and we get:
		\begin{align*}
			\eta_1\star\eta_1 &= -\eta_1 & \\
			\eta_1\star\eta_s &= \tfrac{s-1-n}{n}\eta_s & \text{for }s=2,\dots,n \\
			\eta_s\star\eta_t &= 0 & \text{for } t=1,\dots,n-1 \text{ for } s = 2,\dots,n-1\\
			\eta_s\star\eta_n &= \eta_{s-1} & \text{for } s = 2,\dots,n-1 \\
			\eta_n\star\eta_1 &= 0 &\\
			\eta_n\star\eta_s &= (s-1)\eta_{s-1} & \text{for } s = 2,\dots,n-1 \\
			\eta_n\star\eta_n &= \eta_{n-1} &\\
		\end{align*}
	\paragraph{$\mathbf{L_2}$:} We start with the following expression for the coefficients of the characteristic polynomial.
		\begin{equation*}
			\begin{split}
				&\sigma_1 = -x_1-x_n \\
				&\sigma_2 = (x_1+x_2)x_n \\
				&\vdots \\
				&\sigma_i = (-1)^{i}(x_{i-1}+x_i)x_n^{i-1} \\
				&\vdots \\
				&\sigma_{n-1} = (-1)^{n-1}(x_{n-2}+x_{n-1})x_n^{n-2} \\
				&\sigma_n = (-1)^{n}x_{n-1}x_n^{n-1}. \\
			\end{split}
		\end{equation*}
		With
		$$S_{\chi} = \left(\begin{array}{cccccc}
			-\sigma_1 & 1 &   &   &   & \\
			-\sigma_2 & 0 & 1 &   &   & \\
			-\sigma_3 &   & 0 & 1 &   & \\
			\vdots    &   &   & \ddots & \ddots & \\
			-\sigma_{n-1} &   &   &   & 0 & 1 \\ 
			-\sigma_n & 0 &   & \dots &   & 0 
		\end{array}\right)$$ and $$J:= \left(\frac{\partial\sigma_i}{\partial x_j}\right) =$$ $$\left(\begin{array}{cccccc}
			1 &  &   &   &   & -1\\
			x_n& x_n &  &   &   & x_1+x_2\\
			& -x_n^2  & -x_n^2 &  &   & -2(x_2+x_3)x_n\\
			&   &   & \ddots &  & \vdots\\
			&   &   &  (-1)^{n-1}x_n^{n-2} & (-1)^{n-1}x_n^{n-2} & (-1)^{n-1}(n-2)(x_{n-2}+x_{n-1})x_n^{n-3} \\ 
			&   &   &   &  (-1)^{n}x_n^{n-1} & (-1)^{n}(n-1)x_{n-1}x_n^{n-2}
		\end{array}\right)$$
		we get the following Nijenhuis operator:
		$$\left(\begin{array}{ccccccc}
			x_1 & -x_n &   &   &  &  & -x_2\\
			x_2 & 0 & -x_n &   &  &  & -x_2 - 2x_3\\
			x_3 &   & 0 & -x_n &  &  & -2x_3 -3x_4\\
			\vdots    &   &  &  \ddots & \ddots &  & \vdots\\
			x_{n-2}    &   &   &    &  0 & -x_n & -(n-3)x_{n-2}-(n-2)x_{n-1}\\
			x_{n-1} &   &   &   & & 0 & -(n-2)x_{n-1} \\ 
			0 & 0 &   & \dots &  &   & x_n 
		\end{array}\right). $$
		In this case, the structure relations are:
		\begin{align*}
			\eta_1\star\eta_s &= \eta_s & \text{for }s=1,\dots,n-1 \\
			\eta_1\star\eta_n &= 0 &\\
			\eta_s\star\eta_t &= 0 & \text{for } t=1,\dots,n-1 \text{ for } s = 2,\dots,n-1\\
			\eta_s\star\eta_n &= \eta_{s-1} & \text{for } s = 2,\dots,n-1 \\
			\eta_n\star\eta_1 &= 0 &\\
			\eta_n\star\eta_s &= -(s-1)(\eta_{s-1}+\eta_s) & \text{for } s = 2,\dots,n-1 \\
			\eta_n\star\eta_n &= \eta_{n-1}. & \\
		\end{align*}
		We remark also that these two $n$-dimensional examples presented above are differentially non-degenerate LSA by construction.

\paragraph{$\mathbf{L_3}$ and $\mathbf{L_7}$:}
Looking at the two remaining indecomposable operators of the list in Table 4, we see that the following generalisations are other examples of differentially non-degenerate LSAs of dimension $n$. 
First for $n$ odd, i.e, $n=2k+1$, $k\in\mathbb{N}$, consider  
	$$\left(\begin{array}{cc|cc|c|cc|c|c}
	2x_1-x_n & \pm x_2 &  &  &  &  &  &  & x_n-x_1\\
	x_2      & x_n &  &  &  &  &  &  & \\
	\hline
	 &  & 2x_3-x_n & \pm x_4 &  &  &  &  & x_n-x_3\\
	 &  & x_4      & x_n &  &  &  &  & \\
	\hline
	 \vdots &  &  &  & \ddots &  &  &  & \vdots\\
	 \hline
	 &  &  &  &  & 2x_{2j+1}-x_n & \pm x_{2j+2} &  & x_n-x_{2j+1}\\
	 &  &  &  &  & x_{j+1}  & x_n     &  & \\ 
	\hline
	\vdots &  &  &  &  &  &  & \ddots & \vdots\\
	\hline
	 &  &  &  &  &  &  &  &  x_n  
\end{array}\right)$$
for $j=0,\dots,k-1$. Its structure relations are
\begin{align*}
	&\eta_{2j+1}\star\eta_{2j+1} = 2\eta_{2j+1} \\
	&\eta_{2j+1}\star\eta_{2j+2} = \eta_{2j+2} \\
	&\eta_{2j+1}\star\eta_n = -\eta_{2j+1} \\
	&\eta_{2j+2}\star\eta_{2j+1} = 0 \\
	&\eta_{2j+2}\star\eta_{2j+2} = \pm \eta_{2j+1} \\
	&\eta_{2j+2}\star\eta_n = \eta_{2j+2} \\
	&\eta_n\star\eta_{2j+1} = -\eta_{2j+1}\\
	&\eta_n\star\eta_{2j+2} = 0 \\
	&\eta_n\star\eta_n = \sum_{i=0}^{k}\eta_{2i+1} 
\end{align*}
for $j=0,\dots,k-1$ and it is a differentially non-degenerate LSA. \newline Now let $n$ be even, and for convenience $n=2k+2$, for $k\in\mathbb{N}$. Consider
$$\left(\begin{array}{cc|c|cc|c|c|c}
	2x_1-x_n & \pm x_2 &  &  &  &  &  & x_n-x_1\\
	x_2      & x_n &  &  &  &  &  & \\
	\hline
	\vdots &  & \ddots &  &  &  &  &  \vdots\\
	\hline
	&  &  & 2x_{2j+1}-x_n & \pm x_{2j+2} &  &  & x_n-x_{2j+1}\\
	&  &  & x_{j+1}  & x_n     &  &  &\\ 
	\hline
	\vdots &  &  &  &  & \ddots &  & \vdots\\
	\hline
	 &  &  &  &  &  & 2x_{n-1}-x_{n} & x_n-x_{n-1}\\
	\hline
	&  &  &  &  &  &  &  x_n  
\end{array}\right) $$
for $j=0,\dots,k-1$. Its structure relations coincide to the ones of the operator above (for $n$ odd), but for $j=0,\dots,k$ and $\eta_n\star\eta_n = \eta_n + \sum_{i=0}^{k}\eta_{2i+1}$. This LSA is differentially non-degenerate.
\end{myrem}

Further possible works to do are to find other $n$-dimensional general form of this kind of operators. Moreover one could try to find the classification of the same objects in dimension 4.

\newpage
\section{Appendix: code}
\begin{lstlisting}[language=Python, caption= Function to generate the system of equations ]
	# ring of the coefficients of my polynomials in variables x1,x2,x3
	S.<a,b_11,b_12,b_13,b_31,b_21,b_22,b_23,b_32,b_33,c,alpha11,alpha12,alpha13,alpha21,alpha22,alpha23,alpha31,alpha32,alpha33,alpha41,alpha42,alpha43,alpha51,alpha52,alpha53,alpha61,alpha62,alpha63>=QQ[]
	
	def system_generator(S, a=a, b_11=b_11, b_12=b_12, b_13=b_13, b_31=b_31, b_21=b_21, b_22=b_22, b_23=b_23, b_32=b_32, b_33=b_33,	c=c, alpha11=alpha11, alpha12=alpha12, alpha13=alpha13, alpha21=alpha21, alpha22=alpha22, alpha23=alpha23, alpha31=alpha31, alpha32=alpha32, alpha33=alpha33, alpha41=alpha41, alpha42=alpha42, alpha43=alpha43, alpha51=alpha51, alpha52=alpha52, alpha53=alpha53, alpha61=alpha61, alpha62=alpha62, alpha63=alpha63):
		''' function that takes as input:
			- the ring for the coefficients
			- the parameters
		and returns:
			- the system of equations given by the coefficients
		'''
	
		# definition of the space where I work ,to be able to get easily the coeff.
		R.<x1,x2,x3> = S[]
	
		# definition of the sigmas 
		sigma_1=x1
	
		#sigma_2=a*x1^2 + x2x3           #case 1.1
		#sigma_2=a*x1^2 - x2^2 - x3^2    #case 1.2
		#sigma_2=a*x1^2 + x2^2 + x3^2    #case 1.3
		#sigma_2=a*x1^2 + x2^2           #case 2.1
		#sigma_2=a*x1^2 - x2^2           #case 2.2
		sigma_2=x1*x2                   #case 3
		#sigma_2=x1*x2 + x3^2            #case 4.1
		#sigma_2=x1*x2 - x3^2            #case 4.2
		
		sigma_3=b_11 * x1^3 + 3 * b_12 * x1^2 * x2 + 3 * b_13 * x1^2 * x3 + 3 * b_31 * x1 * x3^2 + 3 * b_21 * x1 * x2^2 + b_22 * x2^3 + 3 * b_23 * x2^2 * x3 + 3 * b_32 * x2 * x3^2 + b_33 * x3^3 + 6 * c * x1 * x2 * x3
	
		# computation of L via formula (2)
		J=jacobian((sigma_1,sigma_2, sigma_3),(x1,x2,x3))
		S=matrix([[-sigma_1, 1, 0],[-sigma_2, 0, 1],[-sigma_3, 0, 0]])
		L=J^(-1)*S*J
		#view(L)
	
		# generation of the system of equations
		Q = J.determinant()
		
		P1 = L[1,0]*Q
		P2 = L[1,1]*Q
		P3 = L[1,2]*Q
		P4 = L[2,0]*Q
		P5 = L[2,1]*Q
		P6 = L[2,2]*Q
	
		coeff1 = ((P1-Q*(alpha11*x1+alpha12*x2+alpha13*x3)).numerator()).coefficients()
		coeff2 = ((P2-Q*(alpha21*x1+alpha22*x2+alpha23*x3)).numerator()).coefficients()
		coeff3 = ((P3-Q*(alpha31*x1+alpha32*x2+alpha33*x3)).numerator()).coefficients()
		coeff4 = ((P4-Q*(alpha41*x1+alpha42*x2+alpha43*x3)).numerator()).coefficients()
		coeff5 = ((P5-Q*(alpha51*x1+alpha52*x2+alpha53*x3)).numerator()).coefficients()
		coeff6 = ((P6-Q*(alpha61*x1+alpha62*x2+alpha63*x3)).numerator()).coefficients()
	
		coeff = coeff1 + coeff2 + coeff3 + coeff4 + coeff5 + coeff6
	
		print('The system consists of', len(coeff), 'equations')
		view(coeff)
	
		return coeff
\end{lstlisting}

\begin{lstlisting}[language=python, caption = Maple solutions - Case 1.1]
	{a = 0, alpha11 = 0, alpha12 = -1/3, alpha13 = 0, alpha21 = 0, alpha22 = 0, alpha23 = 0, alpha31 = 0, alpha32 = 0, alpha33 = 0, alpha41 = 0, alpha42 = 0, alpha43 = -2/3, alpha51 = 0, alpha52 = 3*b_22, alpha53 = 0, alpha61 = 0, alpha62 = 0, alpha63 = 0, b_11 = 0, b_21 = 0, b_22 = b_22, b_23 = 0, b_31 = 0, b_32 = 0, b_33 = 0, c = 0}
	
	{a = 0, alpha11 = 0, alpha12 = -2/3, alpha13 = 0, alpha21 = 0, alpha22 = 0, alpha23 = 0, alpha31 = 0, alpha32 = 0, alpha33 = 3*b_33, alpha41 = 0, alpha42 = 0, alpha43 = -1/3, alpha51 = 0, alpha52 = 0, alpha53 = 0, alpha61 = 0, alpha62 = 0, alpha63 = 0, b_11 = 0, b_21 = 0, b_22 = 0, b_23 = 0, b_31 = 0, b_32 = 0, b_33 = b_33, c = 0}
	
	{a = 1/3, alpha11 = 0, alpha12 = -1/3, alpha13 = 0, alpha21 = -1/3, alpha22 = 0, alpha23 = 0, alpha31 = 0, alpha32 = 0, alpha33 = 3b_33, alpha41 = 0, alpha42 = 0, alpha43 = -1/3, alpha51 = 0, alpha52 = -1/(9b_33), alpha53 = 0, alpha61 = -1/3, alpha62 = 0, alpha63 = 0, b_11 = 1/27, b_21 = 0, b_22 = -1/(27*b_33), b_23 = 0, b_31 = 0, b_32 = 0, b_33 = b_33, c = 1/18}
	
	{a = 0, alpha11 = 0, alpha12 = 0, alpha13 = 0, alpha21 = 0, alpha22 = -3b_23, alpha23 = 0, alpha31 = 0, alpha32 = 0, alpha33 = 0, alpha41 = 0, alpha42 = -9b_23^2, alpha43 = -1, alpha51 = -18b_23^2, alpha52 = 81b_23^3, alpha53 = 9b_23, alpha61 = 0, alpha62 = 3b_23, alpha63 = 0, b_11 = 0, b_21 = -3b_23^2, b_22 = 27b_23^3, b_23 = b_23, b_31 = 0, b_32 = 0, b_33 = 0, c = 0}
	
	{a = a, alpha11 = alpha11, alpha12 = alpha12, alpha13 = alpha13, alpha21 = alpha21, alpha22 = alpha22, alpha23 = alpha23, alpha31 = alpha31, alpha32 = alpha32, alpha33 = alpha33, alpha41 = alpha41, alpha42 = alpha42, alpha43 = alpha43, alpha51 = alpha51, alpha52 = alpha52, alpha53 = alpha53, alpha61 = alpha61, alpha62 = alpha62, alpha63 = alpha63, b_11 = 0, b_21 = 0, b_22 = 0, b_23 = 0, b_31 = 0, b_32 = 0, b_33 = 0, c = 0}
	
	{a = 0, alpha11 = 0, alpha12 = -1, alpha13 = -9b_32^2, alpha21 = 0, alpha22 = 0, alpha23 = 3b_32, alpha31 = -18b_32^2, alpha32 = 9b_32, alpha33 = 81b_32^3, alpha41 = 0, alpha42 = 0, alpha43 = 0, alpha51 = 0, alpha52 = 0, alpha53 = 0, alpha61 = 0, alpha62 = 0, alpha63 = -3b_32, b_11 = 0, b_21 = 0, b_22 = 0, b_23 = 0, b_31 = -3b_32^2, b_32 = b_32, b_33 = 27b_32^3, c = 0}
	
	{a = 1/3, alpha11 = 0, alpha12 = -1/3, alpha13 = 0, alpha21 = -1/3, alpha22 = 1/(64b_32), alpha23 = 3b_32, alpha31 = 0, alpha32 = 9b_32, alpha33 = 192b_32^3, alpha41 = 0, alpha42 = 0, alpha43 = -1/3, alpha51 = 0, alpha52 = -1/(36864b_32^3), alpha53 = -3/(64b_32), alpha61 = -1/3, alpha62 = -1/(64b_32), alpha63 = -3b_32, b_11 = 1/27, b_21 = 0, b_22 = -1/(110592b_32^3), b_23 = -1/(192b_32), b_31 = 0, b_32 = b_32, b_33 = 64*b_32^3, c = 1/18}
	
	{a = 1/4, alpha11 = 0, alpha12 = -3/8, alpha13 = 72b_32^2, alpha21 = -1/4, alpha22 = 1/(192b_32), alpha23 = 3b_32, alpha31 = 144b_32^2, alpha32 = 9b_32, alpha33 = 5184b_32^3, alpha41 = 0, alpha42 = 1/(4608b_32^2), alpha43 = -3/8, alpha51 = 1/(2304b_32^2), alpha52 = -1/(36864b_32^3), alpha53 = -1/(64b_32), alpha61 = -1/4, alpha62 = -1/(192b_32), alpha63 = -3b_32, b_11 = 0, b_21 = 1/(13824b_32^2), b_22 = -1/(110592b_32^3), b_23 = -1/(576b_32), b_31 = 24b_32^2, b_32 = b_32, b_33 = 1728*b_32^3, c = 1/24}
\end{lstlisting}

\begin{lstlisting}[language=python, caption = Analysis of the solutions - Case 1.1]
	## I Solution:
	# 
	# a = 0
	# b_11 = 0
	# b_12 = 0
	# b_13 = 0
	# b_21 = 0
	# b_31 = 0
	# b_23 = 0
	# b_32 = 0
	# b_22 = 0
	# b_33 = b_33 --> choose as 1/3
	# c = 0
	
	x1,x2,x3 = var('x1,x2,x3')
	
	sigma_1=x1
	
	sigma_2=x2*x3
	
	sigma_3= 1/3 * x3^3 
	
	
	## II Solution:
	# 
	# a = 0
	# b_11 = 0
	# b_12 = 0
	# b_13 = 0
	# b_21 = -3*b_23^2 --> -1/3
	# b_31 = 0
	# b_23 = b_23 --> -1/3
	# b_32 = 0
	# b_22 = 27*b_23^3 --> -1 
	# b_33 = 0
	# c = 0
	
	x1,x2,x3 = var('x1,x2,x3')
	
	sigma_1=x1
	
	sigma_2=x2 * x3
	
	sigma_3= -x1*x2^2 - x2^3 - x2^2*x3
	

	## III Solution:
	# 
	# a = 1/3
	# b_11 = 1/27
	# b_12 = 0
	# b_13 = 0
	# b_21 = 0
	# b_31 = 0
	# b_23 = -1/(192*b_32) --> - 1/24
	# b_32 = b_32 --> 1/8
	# b_22 = -1/(110592*b_32^3) --> - 1/216
	# b_33 = 64*b_32^3 --> 1/8
	# c = 1/18
	
	x1,x2,x3 = var('x1,x2,x3')
	
	sigma_1=x1
	
	sigma_2=1/3 * x1^2 + x2 * x3
	
	sigma_3=1/27*x1^3 - 1/216*x2^3 + 1/3*x1*x2*x3 - 1/8*x2^2*x3 + 3/8*x2*x3^2 + 1/8*x3^3
	
	
	## IV Solution:
	# 
	# a = 1/3
	# b_11 = 1/27
	# b_12 = 0
	# b_13 = 0
	# b_21 = 0
	# b_31 = 0
	# b_23 = 0
	# b_32 = 0
	# b_22 = -1/(27*b_33) --> -1/27
	# b_33 = b_33 --> 1
	# c = 1/18
	
	x1,x2,x3 = var('x1,x2,x3')
	
	sigma_1=x1
	
	sigma_2=1/3 * x1^2 + x2 * x3
	
	sigma_3=1/27 * x1^3 -1/27 * x2^3 + x3^3 + 1/3 * x1 * x2 * x3
	
	
	## V Solution:
	# 
	# a = 1/4
	# b_11 = 0
	# b_12 = 0
	# b_13 = 0
	# b_21 = 1/(13824*b_32^2) --> 1/6
	# b_31 = 24*b_32^2 --> 1/96
	# b_23 = -1/(576*b_32) --> -1/12
	# b_32 = b_32 --> 1/48
	# b_22 = -1/(110592*b_32^3) --> -1
	# b_33 = 1728*b_32^3 --> 1/64
	# c = 1/24
	
	
	x1,x2,x3 = var('x1,x2,x3')
	
	sigma_1=x1
	
	sigma_2=1/4 * x1^2 + x2 * x3
	
	sigma_3=1/2*x1*x2^2 - x2^3 + 1/4*x1*x2*x3 - 1/4*x2^2*x3 + 1/32*x1*x3^2 + 1/16*x2*x3^2 + 1/64*x3^3
\end{lstlisting}

\begin{lstlisting}[language=python, caption= Analysis of the solutions - Case 1.2]
	# Solution 1
	#
	# a=1/4
	# b_11 = 0
	# b_12 = 0
	# b_13 = 0
	# b_21 = -1/6
	# b_31 = 0
	# b_22 = 0
	# b_23 = 1/3
	# b_32 = 0
	# b_33 = 0
	# c = 0
	
	x1,x2,x3 = var('x1,x2,x3')
	
	sigma_1=x1
	
	sigma_2=1/4 * x1^2 - x2^2 - x3^2
	
	#sigma_3=b_11 * x1^3 + 3 * b_31 * x1 * x3^2 + 3 * b_21 * x1 * x2^2 + b_22 * x2^3 + 3 * b_23 * x2^2 * x3 + 3 * b_32 * x2 * x3^2 + b_33 * x3^3 + 6 * c * x1 * x2 * x3
	sigma_3=- 1/2 * x1 * x2^2 + x2^2 * x3 
	
	
	# Solution 2 
	
	# RootOf(3*_Z^2 - 1) = p = +- 1/sqrt(3)
	
	# w.l.o.g choose p = 1/sqrt(3) 
	# a linear transformation x3--> -x3 cover the case p = -1/sqrt(3)
	
	# a=1/3
	# b_11 = 1/27
	# b_12 = 0
	# b_13 = 0
	# b_21 = -1/9
	# b_31 = -1/9
	# b_22 = 0
	# b_23 = -p/3 = - 1 / (3*sqrt(3))
	# b_32 = 0
	# b_33 = -(2*p)/3 = - 2 / (3*sqrt(3))
	# c = 0
	
	x1,x2,x3 = var('x1,x2,x3')
	
	sigma_1=x1
	
	sigma_2=1/3 * x1^2 - x2^2 - x3^2
	
	#sigma_3=b_11 * x1^3 + 3 * b_31 * x1 * x3^2 + 3 * b_21 * x1 * x2^2 + b_22 * x2^3 + 3 * b_23 * x2^2 * x3 + 3 * b_32 * x2 * x3^2 + b_33 * x3^3 + 6 * c * x1 * x2 * x3
	sigma_3=1/27 * x1^3 - 1/3 * x1 * x3^2 - 1/3 * x1 * x2^2 - 1 /sqrt(3) * x2^2 * x3 - 2 / (3*sqrt(3)) * x3^3 
	
	
	# Solution 3
	
	# RootOf(3*_Z^2 - 1) = p = +- 1/sqrt(3)
	
	# w.l.o.g. choose p = 1/sqrt(3)
	# a linear transformation x3--> -x3 cover the case p = -1/sqrt(3)
	
	# a=1/3
	# b_11 = 1/27
	# b_12 = 0
	# b_13 = 0
	# b_21 = -1/9
	# b_31 = -1/9
	# b_22 = 0
	# b_23 = +(2*p)/3 = + 2 / (3*sqrt(3))
	# b_32 = 0
	# b_33 = -(2*p)/3 = - 2 / (3*sqrt(3))
	# c = 0
	
	x1,x2,x3 = var('x1,x2,x3')
	
	sigma_1=x1
	
	sigma_2=1/3 * x1^2 - x2^2 - x3^2
	
	#sigma_3=b_11 * x1^3 + 3 * b_31 * x1 * x3^2 + 3 * b_21 * x1 * x2^2 + b_22 * x2^3 + 3 * b_23 * x2^2 * x3 + 3 * b_32 * x2 * x3^2 + b_33 * x3^3 + 6 * c * x1 * x2 * x3
	sigma_3=1/27 * x1^3 - 1/3 * x1 * x3^2 - 1/3 * x1 * x2^2 + 2 /sqrt(3) * x2^2 * x3 - 2 / (3*sqrt(3)) * x3^3 
\end{lstlisting}

\end{document}